\documentclass[12pt]{amsart}

\usepackage{amsfonts,amsmath,amssymb,mathrsfs,amsthm}
\usepackage[dvips]{graphicx,color}
\usepackage[all]{xy}

\newtheorem{theorem}{Theorem}
\newtheorem{proposition}[theorem]{Proposition}

\newtheorem{corollary}[theorem]{Corollary}

\theoremstyle{definition}

\newtheorem{example}[theorem]{Example}

\theoremstyle{remark}
\newtheorem{remark}[theorem]{Remark}

\numberwithin{theorem}{section}
\numberwithin{equation}{section}

\newcommand{\st}{\mathrm{st}}          
\newcommand{\Des}{D}
\newcommand{\Dset}{\operatorname{Des}}

\newcommand{\perm}{\mathfrak{S}}

\newcommand{\pair}[2]{\left\langle #1~|~#2\right\rangle}

\newcommand{\prob}{\mathrm{Prob}}

\newcommand{\comp}{\vDash}              
\newcommand{\Comp}{\operatorname{Comp}}
\newcommand{\co}{\operatorname{co}}

\newcommand{\End}{\mathrm{End}}

\newcommand{\eps}{\lambda}
\newcommand{\Qmap}{\Phi}

\newcommand{\Q}{\mathbb{Q}}
\newcommand{\R}{\mathbb{R}}
\newcommand{\calQ}{\mathcal{Q}}

\newcommand{\QSym}{\calQ}

\newcommand{\Sol}{\mathcal{D}}
\newcommand{\dhat}{\widehat{\Sol}}

\def\field{\Q}

\newcommand{\QSf}{\prob}
\newcommand{\Kbar}{\overline{K}}

\newcommand{\rank}{\operatorname{rank}}

\begin{document}
 \title[Random walks on quasisymmetric functions]{
Random walks on quasisymmetric functions}

\author{Patricia Hersh}
\author{Samuel K. Hsiao}

\date{September 10, 2007}

\address{Department of Mathematics\\
Indiana University\\
Bloomington, IN 47405, USA}
\email{phersh@indiana.edu}

\address{Mathematics Program\\
Bard College\\
Annandale-on-Hudson, NY 12504, USA}
\email{hsiao@bard.edu}

\thanks{The first author was supported in part by NSF grant 0500638.  The second author
was supported in part by an NSF Postdoctoral Research Fellowship}
 

\begin{abstract} 
Conditions are provided under which an endomorphism on quasisymmetric functions gives rise to a left random walk on the descent algebra which is also a lumping of a left random walk on permutations.  Spectral results are also obtained. Several well-studied random walks are now realized this way: Stanley's $QS$-distribution results from endomorphisms given by evaluation maps, $a$-shuffles result from the $a$-th convolution power of the universal character, and the Tchebyshev operator of the second kind introduced recently by Ehrenborg and Readdy yields traditional riffle shuffles. A conjecture of Ehrenborg regarding the spectra for a family of random walks on $ab$-words is proven. A theorem of Stembridge from the theory of enriched $P$-partitions is also recovered as a special case.
\end{abstract}

\maketitle

\section{Introduction}

Quasisymmetric functions have long been used for encoding and manipulating enumerative combinatorial data. They admit a natural graded Hopf algebra structure $\QSym = \bigoplus_{n=0}^\infty \QSym_n$ central to the study of combinatorial Hopf algebras \cite{ABS}.  Their dual relationship to noncommutative symmetric functions \cite{NC1}, as well as to Solomon's descent algebras \cite{Sol,G,MR}, are an important part of the story 
and have inspired a broad literature.

A major goal of this paper is to identify and develop bridges between some of this literature and the body of work surrounding Bidigare, Hanlon, and Rockmore's far-reaching generalizations of Markov chains for various common shuffling and sorting schemes, the main references being \cite{BD, BHR, Brown, Brown2, BrD}.  Stanley first recognized and established a connection between quasisymmetric functions and this work in \cite{QS}.  We deepen this connection and draw combinatorial Hopf algebras into the story by  proving that endomorphisms on quasisymmetric functions which satisfy certain nonnegativity requirements always give rise to random walks both on permutations and also on the descent algebra.  Stanley's $QS$-distribution may be regarded as the special case where the endomorphism corresponds to an evaluation map. Some consequences of this relationship are as follows: (1) whenever a random walk on permutations or the descent algebra arises this way, its transition matrix will be lower triangular with respect to the monomial basis of quasisymmetric functions;  (2) these transition matrices often turn out to be diagonalizable with respect to bases that are quite natural from the viewpoint of quasisymmetric functions;  (3) in some cases, one may directly transfer known spectral results from one setting to the other.   

Our original motivation was to understand the probabilistic behavior of an important endomorphism $\Theta:\QSym\to\QSym$ introduced by Stembridge \cite{Ste} in his development of enriched $P$-partitions. Billera, Hsiao, and van Willigenburg observed \cite{BHW} that $\Theta$ can be represented by a stochastic matrix (after normalization) with unique stationary distribution equal to the distribution of peak sets of random permutations. The associated random walk on peak sets was conjectured to specialize a family of random walks on permutations having uniform stationary distribution, a conjecture which turns out to be correct. 

It was soon pointed out \cite{ABS, BHT} that the $\Theta$-map is dual to a specialization at $q=-1$ of the $A \mapsto (1-q) A$ transformation on noncommutative symmetric functions introduced by Krob, Leclerc, and Thibon  \cite{NC2}. They develop a series of general results about these transformations that include a complete description of their spectral decompositions. They also describe these transformations as acting by multiplication on Solomon's descent algebra, and based on their description it is easy to resolve the conjecture made in \cite{BHW}. The dual relationship between $\Theta$ and a multiplicative operator on the descent algebra turns out to be one instance of a much more general phenomenon, which we develop in Section~\ref{S:random-walks}.

Section ~\ref{bg-section} briefly gives background on quasisymmetric functions, the descent algebra, noncommutative symmetric  functions and lumping of random walks. Section~\ref{S:random-walks} draws together results about characters on combinatorial Hopf algebras and makes some new observations so as  to  characterize  which endomorphisms on quasisymmetric functions give rise to Markov chains. In addition, a characterization of when the stationary distribution  is unique (in which case it is uniform) is given.  In Section~\ref{recursive-section}, it is then shown how in this setting one may read off the eigenvalues, and in fact an eigenbasis is constructed by a  recursive procedure.  Section ~\ref{explicit-matrix-section} describes the resulting transition matrices quite explicitly.

Turning now to applications, Section~\ref{QS-section}  expresses Stanley's $QS$-distribution as the special  case of random walks driven by the endomorphism on $\QSym $ resulting from evaluating the  quasisymmetric functions at a specified  point $(r_1,r_2,\dots )$ where each $r_i$ is a real number.  Section ~\ref{a-section}  deals with the well-studied $a$-shuffles where a deck of cards is split into $a$ (possibly empty) piles which are then shuffled; this random walk results from a very natural endomorphism on $\QSym$, namely the $a$-th convolution power of the universal character.  Connecting all this to the literature on enumeration in posets, we show that Ehrenborg and Readdy's Tchebyshev operator of the second kind on $\QSym$ encodes standard riffle shuffles \cite{ER}, and in Section~\ref{S:ab} we prove a conjecture of Ehrenborg regarding the spectra for a certain family of random walks on $ab$-words that arise from the $r$-Birkoff transform \cite{E04}. In fact, we explicitly describe the transition probabilities of these walks in a way that generalizes a theorem of Stembridge from \cite{Ste}.

For an overview of the literature on shuffling and related topics, the survey by Diaconis \cite{Diaconis} is a very helpful resource. We also recommend \cite{Swe} and \cite{ABS, AM} for further background on Hopf algebras and on combinatorial Hopf algebras, respectively.

%
%

\section{Background}\label{bg-section}

This section reviews background on quasisymmetric functions, noncommutative symmetric functions, and the descent algebra, including bases and products to be used in later sections, as well as a key property of the resulting Markov chains called lumping.  We work over the rational numbers $\field$, although most of our results hold over any field of characteristic $0$. 

\subsection{Quasisymmetric functions} Let $x_1, x_2, \ldots$ be an ordered list of variables. Let $n\ge 0$ and $\alpha = (a_1, \ldots, a_k)$ be a {\it composition} of $n$, that is, a sequence of positive integers that sums to $|\alpha| = n$. The abbreviated notation $\alpha = a_1 \ldots a_k$ will often be used. The {\em monomial  quasisymmetric function} indexed by $\alpha$  is the formal power series
\[ M_\alpha = \sum_{i_1 < \cdots < i_k} x_{i_1}^{a_1}\cdots x_{i_k}^{a_k}. \]
Any linear combination (over $\field$) of monomial quasisymmetric functions is called a {\em quasisymmetric function.}  In order for a formal power series to be a quasisymmetric function, notice that for any composition $\alpha $ and any  two monomials $x_{i_1}^{a_1}x_{i_2}^{a_2} \cdots x_{i_k}^{a_k}$ and $x_{j_1}^{a_1}x_{j_2}^{a_2}\cdots x_{j_k}^{a_k }$ with $i_1 < \cdots < i_k$ and $j_1 < \cdots < j_k $, these monomials must have the same coefficient in the quasisymmetric function; thus, quasisymmetric functions are indexed by the compositions of integers in exactly the way that symmetric functions are indexed by number partitions. They were introduced by Gessel  as generating functions for weights of $P$-partitions \cite{G}.

Let $\Comp(n)$ denote the set of compositions of $n$ and  $\QSym_n$ denote the linear span of $\{M_\alpha\}_{\alpha\in\Comp(n)}$. The vector space $\QSym = \bigoplus_{n\ge 0} \QSym_n$ of quasisymmetric functions has the structure of a graded Hopf algebra: The product is ordinary multiplication of power series and the coproduct is defined on a monomial function by
\[ \Delta_\calQ(M_\alpha) = \sum_{\alpha = \beta \cdot \gamma} M_\beta \otimes M_\gamma,\]
the sum being over all ways of writing $\alpha$ as the concatenation of two (possibly empty) compositions $\beta$ and $\gamma$. The set of (Hopf algebra) endomorphism on $\QSym$ is denoted $\End(\QSym)$.

Define a partial order on $\Comp(n)$ by setting $\alpha \le \beta$ if $\beta$ is a refinement of $\alpha$; that is, $\beta$ is the concatenation of compositions $\beta = \beta_1 \cdots \beta_k$ such that $\alpha = (|\beta_1|, \ldots ,|\beta_k|)$. Given $\alpha = (a_1, \ldots , a_k) \in \Comp(n)$ let us define $S_\alpha\subseteq [n-1]$ by  $S_\alpha = \{ a_1, a_1 + a_2, \ldots, a_1 + \cdots + a_k\}$. The correspondence $\alpha\mapsto S_\alpha$ is an isomorphism of posets between $\Comp(n)$ and the set of subsets of $[n-1]$ under inclusion. The composition corresponding to $S \subseteq [n-1]$ is denoted $\co(S)$. Thus, $\co(S_\alpha) = \alpha$. 

For $\alpha \in \Comp(n)$, define the {\it fundamental quasisymmetric function} $F_\alpha$ by
\begin{equation}\label{E:F-formula}
F_\alpha = \sum_{i_1 \le i_2 \le \cdots \le i_n \atop i_k \in S_{\alpha} \implies i_k < i_{k+1}} x_{i_1} \cdots x_{i_n} = \sum_{\beta\in\Comp(n): \; \beta \ge \alpha} M_\beta.
\end{equation}
By inclusion-exclusion,
\[ M_\alpha = \sum_{\beta\in\Comp(n): \; \beta \ge \alpha} (-1)^{\ell(\beta) - \ell(\alpha)} F_\beta, \]
where $\ell(\alpha)$ denotes the number of parts, or {\it length,} of $\alpha$. Thus $\{F_\alpha\}$ is a basis for $\QSym$.

See \cite{EC2, MR} for further background on quasisymmetric functions.

\subsection{Descent algebra}

For $n\ge 0$, let $\perm_n$ denote the set of permutations of $[n]$. A permutation $\sigma \in \perm_n$ will be represented as a sequence $\sigma = (\sigma_1, \ldots, \sigma_n)$, where $\sigma_i = \sigma(i)$. 
The descent set of $\sigma$ is  defined by $\Dset(\sigma) = \{ i \in [n-1] : \sigma_i > \sigma_{i+1}\}$. The {\em descent composition} of $\sigma$ is defined by $\Des(\sigma) = \co(\Dset(\sigma))$. The set of permutations that have the same descent composition is called a {\it descent class.}

For $\alpha \in \Comp(n)$, define $Y_\alpha\in\Sol_n$ and $X_\alpha\in\Sol_n$ by
\begin{equation}
 Y_\alpha = \sum_{\sigma \in \perm_n: \Des(\sigma) = \alpha} \sigma
 \end{equation}
\begin{equation}\label{E:Xdef}
X_\alpha = \sum_{\beta\in\Comp(n):\beta\le\alpha} Y_\beta = \sum_{\sigma \in \perm_n: \Des(\sigma) \le \alpha} \sigma.
\end{equation}
A well-known result due to Solomon \cite{Sol} asserts that the vector space spanned by  $\{X_\alpha\}_{\alpha\in\Comp(n)}$ (or equivalently $\{Y_\alpha\}_{\alpha\in\Comp(n)}$) is a subalgebra of the group algebra $\field [\perm_n]$.
This subalgebra is called the {\em descent algebra} and we denote it by $\Sol_n$.

Let $\Sol = \bigoplus_{n=0}^\infty \Sol_n$ and $\dhat = \prod_{i=0}^\infty \Sol_i$. An element  $W \in \dhat$ can be represented uniquely as a 
formal series $W = \sum_{n \ge 0} W_n$, where $W_n \in \Sol_n$. 
From this viewpoint $\Sol$ is the subspace of $\dhat$ consisting of those formal series having only finitely many nonzero terms. We will think of $\dhat$ as an algebra with component-wise multiplication:
$\left( \sum W_n \right) \cdot \left(\sum V_n \right) = \sum W_n \cdot V_n.$
Note that the series $X = \sum_{n\ge 0} X_n$ is the identity element of $\dhat$.

Let $\pair{\cdot}{\cdot} : \dhat \times \QSym \to \field$
be the bilinear form defined for any pair of compositions $\alpha,\beta$ by
\[ \pair{X_\alpha}{M_\beta} = \delta_{\alpha,\beta}, \;\;\;\;\text{or equivalently}\;\;\;\;\;
\pair{Y_\alpha}{F_\beta} = \delta_{\alpha,\beta}. \]
We hereby use the pairing $\pair{\cdot}{\cdot}$ to identify
$\dhat$ with the dual vector space of $\QSym$, and $\Sol_n$ with the dual of $\QSym_n$.

\subsection{Noncommutative symmetric functions}
Define a new product $\star$ and coproduct $\Delta_\Sol$ on $\Sol$ by
\begin{equation} \label{E:X-mult}
 X_{\alpha} \star X_\beta = X_{\alpha \cdot \beta}
 \end{equation}
\[ \Delta_\Sol(X_n) = \sum_{ i + j = n} X_i \otimes X_j,\]
The definition of $\Delta_\Sol$ extends to every $X_\alpha$ by requiring $\Delta_\Sol(U\star V) = \Delta_\Sol(U) \star \Delta_\Sol(V)$. There is a natural isomorphism between $(\Sol,\star,\Delta_\Sol)$ and the graded Hopf algebra of noncommutative symmetric functions \cite{NC1}, in which $X_\alpha$ is mapped to the complete symmetric function $S^\alpha$.

The following result is part of Theorem~6.1 in \cite{NC1} (cf.\ \cite{MR}).

\begin{theorem}
$(\Sol,\star,\Delta_\Sol)$ is the graded dual Hopf algebra of $\QSym$ under the pairing $\pair{\cdot}{\cdot}$.
\end{theorem}

Under the isomorphism between $(\Sol,\star,\Delta_\Sol)$ and the Hopf algebra of noncommutative symmetric functions, the usual product on $\Sol$ inherited from the group algebra is opposite to the internal product on noncommutative symmetric functions \cite[Section~5]{NC1}. The next result, which appears as Proposition~5.2 in \cite{NC1}, explicitly relates multiplication in the descent algebra to the product $\star$ and coproduct $\Delta_\Sol$ (and thereby also $\QSym$, by duality). 

\begin{proposition}\label{P:compat}
For $r\ge 2$ let $\Delta_\Sol^{r}$ be defined inductively by $\Delta^{2}_\Sol = \Delta_\Sol$ and $\Delta^{r}_\Sol = \Delta^{r-1}_\Sol\otimes I$, where $I$ is the identity operator on $\Sol$. For any $G, F_1, \ldots, F_r \in \Sol$, we have
\[ G \cdot (F_1 \star \cdots \star  F_r) =   \sum_G (G_{(1)} \cdot F_1) \star \cdots \star (G_{(r)} \cdot F_r).\]
where $\Delta_\Sol^{r}(G) = \sum_G G_{(1)} \otimes \cdots \otimes G_{(r)}$ in Sweedler notation.
\end{proposition}

\subsection{Lumping of random walks}

Suppose that $K$ is the transition probability matrix for a Markov chain with state space $\perm_n$. Suppose also that $K$ has the property that
\begin{equation} \label{E:lumpable}
 \sum_{\sigma\in\perm_n: \Des(\sigma) = \alpha} K(\pi,\sigma) = \sum_{\sigma\in\perm_n: \Des(\sigma) = \alpha} K(\tau,\sigma)
 \end{equation}
for any $\alpha\in\Comp(n)$ and $\pi,\tau\in\perm_n$ such that $\Des(\pi) = \Des(\tau)$. In this case we may define a new Markov chain with state space $\Comp(n)$ and transition probability matrix $\Kbar$ given by 
\begin{equation}\label{E:lump-def}
\Kbar(\Des(\pi),\beta) = \sum_{\sigma\in\perm_n : \Des(\sigma) = \beta} K(\pi, \sigma).
\end{equation}
In other words, $\Kbar(\alpha,\beta)$ is the probability that a permutation with descent composition $\alpha$ moves to some permutation with descent composition $\beta$ in one step of the original Markov chain. We shall say that $\Kbar$ {\em lumps $K$ by descent sets}. Lumping is discussed in \S6.3 in Kemeny and Snell's book \cite{KS}.

Lumping by descent sets occurs when a random walk on $\perm_n$ is driven by a probability distribution that is constant on descent classes, as we now explain. Let $W = \sum_{\sigma\in\perm_n} W(\sigma) \sigma$ be a probability distribution on $\perm_n$ such that $W \in \Sol_n$. Consider the Markov chain with state space $\perm_n$ and transition probability matrix given by $K(\pi, \sigma \pi) = W(\sigma^{-1}).$ Then the expression on the left-hand side of \eqref{E:lumpable} represents the coefficient of $\pi$ in $W\cdot Y_\alpha$ while the right-hand side represents the coefficient of $\tau$. If $\Des(\pi) = \Des(\tau)$ then these coefficients are equal by the fact that $\Sol_n$ is a subalgebra of $\field[\perm_n]$. This means $K$ can be lumped by descent sets, and the column of $\Kbar^m$ indexed by $(n)$ encodes the $m$th convolution power of the distribution $W$:
\begin{equation}\label{E:Kbar-convolve}
 \Kbar^m(\Des(\pi), n) = W^{*m}(\pi).
 \end{equation}

Now we show how the eigenvalues and eigenvectors of the two matrices are related to each  other.

\begin{proposition}
Each eigenvector of $\Kbar$ 
gives rise to an eigenvector for $K$ with the same eigenvalue.
Moreover,  linearly independent eigenvectors are thereby sent to eigenvectors that 
remain linearly independent.
\end{proposition}

\begin{proof}
If the coordinate indexed by a particular descent class has value 
$a_i$ in a chosen eigenvector of $\Kbar$, then assigning value 
$a_i/d$  for each permutation in the descent class, letting $d$ be the size of the descent class,
yields an eigenvector of $\Kbar$.  
The fact that this is indeed an eigenvector of $K$ follows again from the fact that $K$ admits
a lumping according to descent classes.
\end{proof}

\begin{proposition}
The matrix $K$ is block diagonalizable, with one $(n! - 2^{n-1})$ by $(n! -2^{n-1})$
block and each descent class $D_i$
giving rise to a 1 by 1 block.
\end{proposition}

\begin{proof}
The idea is to show how to decompose the vector space upon which $K^T$ acts into  subspaces each of which is carried to itself by $K^T$.   For each descent class $D_i$, notice that $K^T$ has an eigenvector by setting each coordinate indexed by an element of the  descent class to 1 and all other coordinates to 0.  On the other hand, we obtain an $(n! - 2^{n-1})$-dimensional subspace also sent to itself by $K^T$ by considering those vectors where
the coordinates indexed by permutations in any chosen descent class
$D_i$ sum to 0 and all other coordinates are 0.
It follows easily from the definition of lumping that each of these vectors is sent to a vector with
the property that for any descent class, its coordinates which are indexed by permutations in
that descent class sum to 0.
\end{proof}
 
\noindent
{\bf Question:} When our random walk arises from a left action of the descent algebra, does this imply further  structure on $K$?


\section{Random walks resulting from endomorphisms of $\QSym$}
\label{S:random-walks}

This section ties together results from \cite{ABS}, \cite{NC1}, and \cite{R} to deduce a new consequence, namely that endomorphisms of quasisymmetric functions give rise to random walks under mild nonnegativity conditions.

Let $\Qmap\in\End(\QSym)$ and $n\ge 0$.  Define $c_{\alpha,\beta}$, for $\alpha,\beta\in\Comp(n)$ by
\begin{equation}\label{E:cdef}
\Qmap(F_\alpha) = \sum_{\beta\in\Comp(n)} c_{\alpha,\beta} F_\beta.
\end{equation}
Thus, $c_{\alpha,\beta} = \pair{Y_\beta}{\Qmap(F_\alpha)}.$ 

Suppose that the numbers $c_{\alpha,n}, \alpha\in\Comp(n)$, are nonnegative and identically zero. Define a probability measure $\QSf_\Qmap: \perm_n \to \R$ by 
\begin{equation}\label{E:QSf-distribution}
\QSf_\Qmap(\pi) = \frac{c_{\Des(\pi),n}}{\displaystyle{\sum_{\sigma\in\perm_n}} c_{\Des(\sigma),n}}.
\end{equation}
We will call this the {\em $QS^*$-distribution corresponding to $\Qmap$} because of its connection to Stanley's $QS$-distribution \cite{QS}, as explained in Section~\ref{QS-section}. 

Consider the random walk on $\perm_n$ where $\pi$ goes to $\sigma \pi$ with probability $\QSf_\Qmap(\sigma^{-1})$. Denote the corresponding transition probability matrix by $K$, so that
\[ K (\pi, \sigma \pi) = \QSf_\Qmap(\sigma^{-1}).\]
Note that the transpose of $K$ is the transition matrix of the left random walk on $\perm_n$ driven by $\sum_{\sigma\in\perm_n} \QSf_\Qmap(\sigma) \sigma$, where one step takes $\pi$ to $\sigma \pi$ with probability $\QSf(\sigma)$.

Let $\lambda \in \field$ be defined by $\Qmap(M_1) = \eps M_1$. Let $\Qmap_n$ denote the restriction map $\Qmap|_{\QSym_n}$. If $\eps \ne 0$, then define $\Kbar$ to be the transpose of the matrix representing $\frac{1}{\eps^n} \Qmap_n$ relative to the fundamental basis, namely
\[ \Kbar(\alpha,\beta) = \frac{1}{\eps^n} c_{\alpha,\beta}.\]
for all $\alpha,\beta\in\Comp(n)$. If we need to be explicit about about $n$ and $\Qmap$ then we will write $\Kbar^\Qmap_n$. The same goes for $K$.

\begin{theorem}\label{T:end-markov}
Let $\Qmap\in \End(\QSym)$ and $n\ge 0$, and suppose that the numbers $c_{\alpha,n}, \alpha\in\Comp(n)$, are nonnegative and not identically zero. Then  
$\eps^n > 0$ and $\Kbar$ is a stochastic matrix. Furthermore, $\Kbar$ lumps $K$ by descent sets.
\end{theorem}

Let us illustrate Theorem~\ref{T:end-markov} with an example with
Stembridge's $\Theta$-map before turning to the proof. 

\begin{example} \label{Stembridge-example}
Given $\alpha\in\Comp(n)$ for some 
$n\ge 1$, let $\Lambda(\alpha) = \{ i \in S_\alpha ~|~ i \ne 1$ and $ i -1 \ne S_\alpha\}$. For instance, if $\alpha = 1121134$ then $S_\alpha = \{1,2,4,5,6,9\}$ and $\Lambda(\alpha) = \{4,9\}.$ Thus, if $\alpha$ is the descent composition of a permutation $\pi\in \perm_n$, then $\Lambda(\alpha)$ is the {\em peak set} of $\pi$, namely the set $\{ i \in [2,n-1]~|~\pi_{i-1} < \pi_i > \pi_{i+1}\}.$  According to \cite[Proposition~3.5]{Ste}, $\Theta$ can be defined in terms of the fundamental basis by
\[ \Theta(F_\alpha) = 2^{|\Lambda(\alpha)| + 1} \sum_{\beta\in\Comp(n) \atop \Lambda(\alpha) \subseteq S_\beta \triangle (S_\beta+1)} F_\beta\]
where $T + 1 = \{t + 1 ~|~ t\in T\}$ and $\triangle$ stands for the symmetric difference: $A \triangle B = (A - B) \cup (B - A)$. It is well known that $\Theta$ is a Hopf algebra homomorphism \cite[Example~4.9]{ABS}.

The corresponding $QS^*$-distribution is given by
\begin{equation}\label{E:Stemb-QS}
 \QSf_\Theta(\pi) = \begin{cases} \frac{1}{2^{n-1}} & \text{if the peak set of $\pi$ is empty}\\
0 & \text{otherwise.}\end{cases}
\end{equation}
This is the probability that a deck of $n$ cards is in arrangement $\pi$ after an inverse face-up face-down shuffle \cite{BD}: Remove a subset of cards from the deck, letting all subsets have the same chance of being selected, and place the packet face down on top of the remaining cards. The row of $\Kbar$ indexed by $(n)$ gives the distribution of descent sets after performing one face-up face-down shuffle: Cut the deck into two packets according to the binomial distribution, flip the top packet over so the cards are facing up, then shuffle the two packets; the probability of ending up with a permutation with descent composition $\alpha$ is $\Kbar(n,\alpha)$. For instance,
\[ \Kbar^\Theta_3 = \;\;\;
\begin{tabular}{c|cccc}
    & $3$ & $12 $ & $21$ & $111$ \\
    \hline 
    $3$ &  1/4 & 1/4 & 1/4 & 1/4 \\
       $12$ & 1/4 & 1/4 & 1/4 & 1/4  \\
      $21$ & 0 & 1/2 & 1/2 & 0  \\
      $111$ & 1/4 & 1/4 & 1/4 & 1/4
      \end{tabular}
\]

Each entry, say $\Kbar(21,12) = 1/2$, can be explained in terms of lumping $K$ by descents as follows. Pick any permutation $\pi\in\perm_3$ such that $\Des(\pi) = 21$, say $\pi = {\bf 132}$ (short for $\pi_1 = 1, \pi_2 = 3, \pi_3 = 2$). Then $\Kbar(21,12)$ should be the probability a permutation $\sigma\in\perm_3$ chosen with probability $\QSf_\Theta(\sigma)$ will have the property that $\Des(\sigma^{-1} \pi) = 12$. 

Stembridge  \cite[Theorem~3.6]{Ste} provides an alternate probabilistic interpretation of $\Theta$: Given $\pi\in\perm_n$, independently assign a $+$ or $-$ sign to every $\pi_i$, letting each sign occur with probability $1/2$. Then the probability that the resulting signed permutation has descent composition $\beta$ is $\Kbar(\Des(\pi),\beta)$. We generalize Stembridge's result in Section~\ref{S:ab}.
\end{example}

The proof of Theorem~\ref{T:end-markov} amounts to showing that $\Qmap_n$ is, after normalization, dual to the operator on $\Sol_n$ given by $W\mapsto\left(\sum_{\sigma\in\perm_n} \QSf_\Qmap(\sigma) \, \sigma \right) \cdot W.$ 
We begin by setting up a bijection between $\End(\QSym)$ and those series in $\widehat{\Sol}$ that encode characters of $\QSym$. Recall that we have identified the dual vector space of $\QSym$ with $\dhat$, the set of formal infinite series $\sum_{n \ge 0} W_n$ such that $W_n \in \Sol_n$. Thus, a character of $\QSym$ is a series $W\in \widehat{\Sol}$ such that $\pair{W}{1} = 1$ and 
$\pair{W}{FG} = \pair{W}{F} \pair{W}{G}$ for all $F,G\in\QSym$. 

A character of fundamental importance is the series $X = \sum_{n\ge 0} X_n$, where $X_n$ is the identity permutation in $\perm_n$  (see \eqref{E:Xdef}).
This is called the {\it universal character}.
 It follows from the definition of $X$ that 
\[\pair{X}{M_\alpha} = \pair{X}{F_\alpha} = \begin{cases}
1 & \text{if $\alpha = ( ) $ or $\alpha = (n)$}\\
0 & \text{otherwise.}
\end{cases}
\]
Thus, for any quasisymmetric function $F(x_1, x_2, \ldots)$, 
\[\pair{X}{F(x_1,x_2,\ldots)} = F(1,0,0,\ldots).\]
For $\Qmap\in\End(\QSym)$ and $W\in\dhat$, define the series $W^\Qmap \in \dhat$ by 
\[\pair{W^\Qmap}{\cdot} = \pair{W}{\Qmap(\cdot)}.\]
In the theory of combinatorial Hopf algebras, $X$ satisfies a universal property \cite[Theorem~4.1]{ABS}, of which the following is an immediate corollary:

\begin{proposition}
 \label{P:universal}
The correspondence $\Qmap \mapsto X^\Qmap$ is a bijection between $\End(\QSym)$ and the set of characters of $\QSym$. Moreover,
\begin{equation}\label{E:univ}
 \Qmap(M_\alpha) = \sum \pair{X^\Qmap}{M_{\beta_1}} \cdots \pair{X^\Qmap}{M_{\beta_m}} M_{(|\beta_1|, \ldots, |\beta_m|)},
\end{equation}
where the sum is over all sequences of compositions $\beta_1, \ldots, \beta_m$ such that $\beta_1 \cdots \beta_m = \alpha$.
\end{proposition}

\begin{remark}\label{R:triangular}
It follows from Proposition~\ref{P:universal} that the matrix relative to the monomial basis for $\Qmap$ as a linear operator on $\QSym_n$ is lower triangular, provided that the ordering of the basis elements is a linear extension of the partial ordering on $\Comp(n)$.
\end{remark}

Next we develop some properties of characters to be used shortly. The following characterization is proven in \cite[Theorem~3.2(ii)-(iii)]{R}. 

\begin{proposition}\label{P:grouplike}
A series $W\in \widehat{\Sol}$ is a character of $\QSym$ if and only if  $\Delta_\Sol(W_n) = \sum_{i=0}^n W_i \otimes W_{n-i}$ for all $n$ (i.e., $W$ is group-like for $\Delta_\Sol$).
\end{proposition}

We now derive a useful formula for left multiplication in $\Sol$ by $X^\Qmap$. 

\begin{proposition} \label{P:leftmult}
For $\Qmap\in \End(\QSym)$ and $\alpha \in \Comp(n)$, 
\begin{equation}\label{E:leftmult}
X^\Qmap \cdot X_\alpha 
= (X_\alpha)^\Qmap.
\end{equation}
\end{proposition}
\begin{proof}
Let $\alpha = a_1 \ldots a_k \in \Comp(n)$. By applying Proposition~\ref{P:grouplike} and then Proposition~\ref{P:compat}, we obtain
\[
 X^\Qmap \cdot X_\alpha
=  X^\Qmap \cdot (X_{a_1} \star \cdots \star X_{a_k}) \\
 =  (X^\Qmap \cdot X_{a_1}) \star \cdots \star (X^\Qmap \cdot X_{a_k}).
\]
Since $X_{a_i}$ is the identity permutation in $\perm_{a_i}$, we have $X^\Qmap \cdot X_{a_i} = (X^\Qmap)_{a_i}$. For $\beta\in\Comp(n)$,
\[
\pair{ (X^\Qmap)_{a_1} \star \cdots \star (X^\Qmap)_{a_k}}{M_\beta} =
\pair{ (X^\Qmap)_{a_1} \otimes \cdots \otimes (X^\Qmap)_{a_k}}{\Delta_\calQ^{k}(M_\beta)}.
\]
The right-hand side vanishes if $\beta \not\ge \alpha$. If $\beta\ge \alpha$, then there exist compositions $\beta_1, \ldots, \beta_k$ such that $\beta=\beta_1 \cdots \beta_k$ and $a_i = |\beta_i|$, and we have
\[\pair{ (X^\Qmap)_{a_1} \otimes \cdots \otimes (X^\Qmap)_{a_k}}{\Delta_\calQ^{k}(M_\beta)} = 
\pair{X^\Qmap}{M_{\beta_1}} \cdots \pair{X^\Qmap}{M_{\beta_k}}.\] 
According to \eqref{E:univ}, this is the coefficient of $M_\alpha$ in the monomial basis expansion of $\Qmap(M_\beta)$. This coefficient is given by $\pair{X_\alpha}{\Qmap(M_\beta)}$, or  equivalently $\pair{(X_\alpha)^\Qmap}{M_\beta}$.

We have proved that $X^\Qmap \cdot X_\alpha$ and $(X_\alpha)^\Qmap$ agree on the monomial basis, and hence that they are equal.
\end{proof}

\begin{proposition}\label{P:transpose}
If $\Qmap\in \End(\QSym)$ then the character $X^\Qmap$ satisfies
\begin{equation}\label{E:transpose}
\pair{X^\Qmap \cdot W}{ G }= \pair{W}{\Qmap(G)}
\end{equation}
for all $W\in \widehat{\Sol}$ and $G\in\QSym$. 
\end{proposition}
\begin{proof}
By Proposition~\ref{P:leftmult}, \eqref{E:transpose} holds for $W = X_\alpha$ for all $\alpha$, hence it holds in general by linearity.
\end{proof}

\begin{remark}
It is shown in \cite[Lemma~5.1]{Sch} that if $W \in \dhat$ is group-like for $\Delta_\Sol$ then the operator on $\Sol$ given by $V \mapsto W\cdot V$ is a Hopf algebra endomorphism of $(\Sol, \star, \Delta_\Sol)$. This assertion and its converse are direct consequences of Proposition~\ref{P:transpose} (together with Proposition~\ref{P:grouplike}).
\end{remark}

The preceding result says that $X^\Qmap$, as an operator on $\Sol$ acting by left multiplication, is dual to $\Qmap$. Notice that
$\pair{X^\Qmap}{F_\alpha} = \pair{X^\Qmap \cdot Y_n}{F_\alpha} = \pair{Y_n}{\Qmap(F_\alpha)} = c_{\alpha,n},$
so that the homogeneous component of $X^\Qmap$ of degree $n$ is
\begin{equation}\label{E:Xphi-c}
(X^\Qmap)_n =  \sum_{\alpha\in\Comp(n)} c_{\alpha,n} Y_\alpha = \sum_{\sigma\in\perm_n} c_{\Des(\sigma),n} \, \sigma.
\end{equation}

\begin{example}
For the $\Theta$-map, we have 
\[ (X^\Theta)_n = 2 \cdot (Y_n + Y_{(1,n-1)} + Y_{(1^2,n-2)} + \cdots + Y_{1^n}) = 2 \cdot \sum_{\sigma\in\perm_n : \Lambda(\sigma) = \emptyset} \sigma \; , \]
where $\Lambda(\sigma)$ stands for the peak set of $\sigma$. 
The fact that $\Theta$ is dual to $X^\Theta$ was noticed in \cite[Remarks~7.11]{ABN} and \cite{BHT}.
\end{example}

Proceeding with the proof of Theorem~\ref{T:end-markov},
we have
\[
\sum_{\beta\in\Comp(n)} c_{\alpha,\beta} Y_\beta =  X^\Qmap \cdot Y_\beta =
 \sum_{\sigma,\tau\in\perm_n:\Des(\tau) = \beta} c_{D(\sigma),n} \cdot \sigma\tau = \sum_{\pi,\tau\in\perm_n:\Des(\tau) = \beta} c_{D(\pi\tau^{-1}),n} \cdot \pi,
\]
which leads to the identity
 \begin{equation}\label{E:L-identity}
  c_{\alpha,\beta} = \sum_{\tau\in\perm_n: \Des(\tau) = \beta} c_{\Des(\pi\tau^{-1}), n}.
\end{equation}
where $\pi$ is any permutation  such that $\Des(\pi) = \alpha$
In particular, $c_{\alpha,\beta}\ge 0$ for all $\alpha,\beta\in\Comp(n)$. One consequence of \eqref{E:L-identity} is $\sum_{\beta \in\Comp(n)} c_{\alpha,\beta} = \sum_{\tau\in\perm_n} c_{\Des(\pi\tau^{-1}), n} = \sum_{\sigma\in\perm_n} c_{\Des(\sigma), n}$, which means every row of the matrix $c$ has the same (positive) sum. 

Let us show that the sum of the row indexed by $1^n$ is $\eps^n$. We have
\[
 \Qmap(M_{1^n}) = \Qmap(F_{1^n}) = \sum_{\alpha\in\Comp(n)} c_{1^n,\alpha} F_\alpha  = \sum_{\alpha\in\Comp(n)} c_{1^n,\alpha} \sum_{\beta\ge \alpha} M_\beta.
\]
The coefficient of $M_{1^n}$ in the last expression is $\sum_{\alpha\in\Comp(n)} c_{1^n,\alpha}.$ On the other hand, by \eqref{E:univ} this coefficient is $\pair{X^\Qmap}{M_{1}}^n = \eps^n$. 

An interesting consequence is the identity
\begin{equation}
\sum_{\sigma\in\perm_n} c_{\Des(\sigma),n} = \eps^n,
\end{equation}
which in turn implies $\eps^n > 0$ and
\begin{equation}\label{E:QSf-simplified}
\QSf_\Qmap(\pi) =  \Kbar(\Des(\pi),n).
\end{equation}

To complete the proof of Theorem~\ref{T:end-markov}, it remains to show that $\Kbar$ lumps $K$ by descents. This requires checking that for any $\beta\in\Comp(n)$,
\begin{equation}\label{E:lump-identity}
 \sum_{\sigma\in\perm_n : \Des(\sigma) = \beta} c_{\Des(\pi\sigma^{-1}),n} = \sum_{\sigma\in\perm_n: \Des(\sigma) = \beta} c_{\Des(\tau\sigma^{-1}),n}
 \end{equation}
for all $\pi,\tau\in\perm_n$ such that $\Des(\pi) = \Des(\tau)$, and 
\[ \frac{1}{\eps^n} c_{\Des(\pi),\beta} = \sum_{\sigma\in\perm_n: \Des(\sigma) = \beta} K (\pi,\sigma).\]
Both of these identities follow directly from \eqref{E:L-identity}, completing the proof.

\subsection{Stationary distribution}

\begin{theorem}\label{uniform-dist}
Suppose that the hypotheses of Theorem~\ref{T:end-markov} are satisfied and  that $c_{\alpha,n} > 0$ for some $\alpha\in\Comp(n)\setminus \{n, 1^n\}$. Then $K$ has a unique stationary distribution given by the uniform distribution on permutations, and $\Kbar$ has a unique stationary distribution equal to the distribution of descent sets in $\perm_n$.
\end{theorem}
\begin{proof}
The proof has two parts, namely show (1) that for $w \in \R[\perm_n]$ in which each permutation has nonnegative coefficient between 0 and 1 so that  these coefficients add up to 1, the Markov chain resulting from left action by $w$ has a one dimensional space of fixed vectors spanned by the vector in which all coordinates are equal, provided that the set of permutations  appearing with nonzero coefficient in $w$ generate $\perm_n$, and (2) that any single descent  class other than those containing only the identity or only the longest element in $\perm_n$ will  meet this condition of generating $\perm_n$. Our added assumption about $c_{\alpha,n}$ implies that there exists some $\tau\in\perm_n$, which is neither identity nor the longest permutation, such that $\QSf_\Qmap(\tau)> 0$. Since $K^T$ is the matrix for the left action of $\sum_{\sigma\in\perm_n} \QSf_\Qmap(\sigma) \sigma$ on $\perm_n$, it follows from (1) and (2) that $K^T$, and hence $K$, has unique stationary distribution equal to the uniform distribution on permutations. Our assertion about $\Kbar$ follows from elementary properties of lumping.

First we prove (1). Each permutation $\sigma $
appearing with nonzero coefficient $a_{\sigma }$ in $w$ gives a left 
action on $\perm_n$, so for each such $\sigma $
we make a directed graph with $n!$ vertices, using directed edges to 
indicate for each element of $\perm_n$ where it is sent by $\sigma$.   Now combine these graphs for the various $\sigma$ with
nonzero $a_{\sigma }$, i.e. using one vertex set of size $n!$ and the union of all directed edges
for all such $\sigma $.  Choosing a set of $\sigma $ that generate $\perm_n$ implies that this directed
graph has a directed path from each of its $n!$ vertices to all its other vertices.
For $u$ any eigenvector with eigenvalue 1, we have  
$w u = u$, which translates to an equation for each of the $n!$ coordinates in $u$.  
These equations are indexed by permutations, so consider the equation indexed by 
some $\gamma \in \perm_n$.   For $u= \sum_{\sigma \in \perm_n} b_{\sigma } \sigma$, this equation
may be written as 
\[ b_{\gamma } = \sum_{\sigma\in \perm_n}  a_{\sigma }b_{\sigma^{-1}(\gamma )},\]
namely a sum in which the nonzero terms exactly come from permutations at the tails of 
arrows with head at $\gamma $.  This 
expresses $b_{\gamma }$ as a convex combination (i.e. a weighted average) 
of these coefficients $b_{\sigma^{-1}(\gamma )}$, implying $b_{\gamma }$ is neither 
the smallest nor largest value among these coefficients unless all are 
equal.  However, chasing the directed graph around, we see that every permutation is expressed
as a convex combination of others in such a way that no coefficient
$b_{\gamma }$ may be smallest or largest among these coefficients, hence all
must be equal, so we are done. 
If the permutations do not generate $\perm_n$, then the graph will have multiple 
components (resulting from multiple orbits in the left action on $\perm_n$), and we get a basis for the eigenspace with eigenvalue 1 by taking as basis vectors the sums over permutations in any one orbit.

Now to (2).   Consider a permutation $\sigma $ in our chosen descent class $D$.  By assumption, there must be some values $j,j+1$ not appearing consecutively in the one-line notation expression for $\sigma $, this implies $(j,j+1)\sigma \in D$, implying $(j,j+1)$ is in the subgroup $G_D$  of $\perm_n$
generated by the elements of $D$.  Thus, we get all adjacent transpositions except those
$(i,i+1)$ for $i,i+1$ in consecutive positions in $\sigma $, but then consider a maximal segment of
consecutive values appearing consecutively in $\sigma $.  This segment must either be strictly 
increasing or strictly decreasing, and without loss of generality assume the former.  Thus, the 
segment takes the form $i,i+1,\dots ,i+r$, and by assumption we either have $i\ne 1$ or have 
$i+r\ne n$.  Thus, we may swap either the values $i-1,i$ or the values $i+r,i+r+1$ to obtain another
permutation $\pi'$  in our descent class which has strictly shorter segment and which 
may be used to show either that the adjacent transposition
$(i,i+1)$ is in $G_D$ or else that $(i+r,i+r+1)\in G_D$.   Continuing in this manner, one may show
that all adjacent transpositions are in $G_D$, implying $G_D = \perm_n$, as desired.
\end{proof}

\subsection{Lumping by peak sets and other statistics}
We establish variants of Theorem~\ref{T:end-markov} and Theorem~\ref{uniform-dist} from which the probabilistic interpretation of the $\Theta$-map given in \cite{BHW} can be deduced.
  
 \newcommand{\alhat}{\widehat{\alpha}}
 \newcommand{\behat}{\widehat{\beta}}
 \newcommand{\gahat}{\widehat{\gamma}}
 \newcommand{\Khat}{\widehat{K}}
  
Let $\Qmap\in \End(\QSym)$ and $n\ge 0$. Let us say that $\Qmap_n$ has the {\it partitioning property} if the following two conditions hold. First is that there exists an equivalence relation on $\Comp(n)$ such that $\Qmap(F_\alpha) = \Qmap(F_\beta)$ whenever $\widehat{\alpha} = \widehat{\beta}$, where $\widehat{\alpha}$ denotes the equivalence class containing $\alpha$. Let $E$ denote the set of equivalence classes and let $\phi_{\widehat{\alpha}} = \Qmap(F_\alpha)$ for all $\alhat\in E$. The second condition is that $\{\phi_{\widehat{\alpha}}\}_{\alhat\in E}$ is a basis for $\Theta(\QSym_n)$. If $\Qmap_n$ has full rank then it trivially has the partitioning property. A nontrivial example is the $\Theta$-map: the equivalence class containing $\alpha \in \Comp(n)$ is given by $\alhat = \{ \beta\in\Comp(n)~|~\Lambda(\beta) = \Lambda(\alpha)\}.$ An open question is whether there are any other nontrivial endomorphisms with the partitioning property.

Suppose that $\Qmap_n$ has the partitioning property. Then define $d_{\alhat,\behat} \in \field$ by 
\[ \Qmap(\phi_{\alhat})  =  \sum_{\behat \in E} d_{\alhat, \behat} \; \phi_{\behat}\]
and if $\eps^n \ne 0$ as well then define
\[\widehat{K}(\alhat, \behat) = \frac{1}{\eps^n} d_{\alhat, \behat}.\] 
For $\behat \in E$ define $Y_{\behat} \in \Sol_n$ by
\[Y_{\behat} = \sum_{\alpha\in\behat} Y_\alpha.\]
It follows from the duality between $\Qmap$ and $X^\Qmap$ (Proposition~\ref{P:transpose}) that
\begin{equation} \label{E:left-mult-eqclass} X^{\Qmap} \cdot Y_{\behat} = \sum_{\alhat\in E} d_{\alhat,\behat} Y_{\alhat}
\end{equation}
A direct consequence is the following:
\begin{proposition}  \label{P:ideal}
If $\Qmap_n$ has the partitioning property then the subspace of $\Sol_n$ spanned by $\{Y_{\alhat}\}_{\alpha\in E}$ is the right ideal $X^{\Qmap} \cdot \Sol_n$.
\end{proposition}
Specializing to the $\Theta$-map, $X^\Theta\cdot\Sol_n$ is just Nyman's peak algebra \cite{N}, which was shown to be a right ideal in \cite{Sch, ABN} and implicitly in \cite{NC2}.  

Further implications of \eqref{E:left-mult-eqclass} are the following variants of Theorems~\ref{T:end-markov} and \ref{uniform-dist}:
\begin{theorem}\label{T:peak-lump}
Suppose that $\Qmap_n$ has the partitioning property and the numbers $c_{\alpha, n}, \alpha\in\Comp(n)$ are nonnegative and not identically zero. Then $\widehat{K}$ is a stochastic matrix and it lumps $K$ by equivalence classes.
\end{theorem}

\begin{theorem} $\widehat{K}$ has a unique stationary distribution if and only if there is some $\alpha\in\Comp(n) \setminus \{n, 1^n\}$ such that $c_{\alpha,n}\ne 0$. In this case the stationary distribution is the lumped version of the uniform distribution on permutations.
\end{theorem}

These results can be interpreted for the $\Theta$-map as follows, elucidating the results of \cite[\S3.1]{BHW}.
The map $\frac{1}{2^n}\Theta|_{\Pi_n}$ is dual to the left action of  $\sum_{\sigma\in\perm_n} \QSf_\Theta(\sigma) \sigma = \frac{1}{2^{n-1}} \sum_{\sigma\in\perm_n: \Lambda(\sigma) = \emptyset} \sigma $ on the right ideal $X^\Theta \cdot \Sol_n$. Thus $\frac{1}{2^n}\Theta|_{\Pi_n}$ gives rise to a random walk on peak sets in which the peak set of a permutation $\pi$ steps to the peak set of $\sigma^{-1}\pi$ with probability $\QSf_\Theta(\sigma)$. The (unique) stationary distribution is the distribution of peak sets of random permutations.


\section{Spectral decomposition}\label{recursive-section}

This section describes the eigenvalues and gives a recursive process for constructing a basis of eigenvectors for all the nonzero eigenspaces. 
See Macdonald's book \cite[Ch.\ VI Section 4]{Mac} for an earlier instance of this type of recursive eigenvector construction arising in a different context. We show that our eigenvectors are primitive elements. 

\subsection{Eigenvalues and eigenvectors}

Given $\Qmap\in \End(\QSym)$, let
\[\eps_n = \pair{X^\Qmap}{M_n} \;\;\; \text{ and } \;\;\;
\eps_{\alpha} = \eps_{a_1} \cdots \eps_{a_k}\]
for every $n\ge 0$ and every composition $\alpha = (a_1, \ldots, a_k)$. By \eqref{E:univ}, $\eps_\alpha$ is the coefficient of $M_\alpha$ in the monomial expansion of $\Qmap(M_\alpha)$ and in particular $\eps_1 = \eps$. 

Since $\Qmap$ is triangular relative to the monomial basis (Remark~\ref{R:triangular}), we obtain the following description of its eigenvalues.

\begin{proposition}\label{P:eigenvalues}
For any $\Qmap\in\End(\QSym)$ and $n\ge 0$, the eigenvalues of $\Qmap_n$, taking into account multiplicities, are $(\eps_\alpha)_{\alpha\in\Comp(n)}$. 
\end{proposition}

Conversely, by Theorem~\ref{T:matrix} some of the eigenvalues can be arbitrarily prescribed:
\begin{proposition}\label{P:u-eigen}
For any list of real numbers $u_1, u_2, u_3, \ldots$, there exists $\Qmap\in\End(\QSym)$ such that $\lambda_i = u_i$ for all $i$.
\end{proposition}

The main theorem of this section is as follows:

\begin{theorem}\label{T:eigen}
Let $\Qmap\in\End(\QSym)$ and $n\ge 0$. Suppose that for every $m\le n$ such that $\eps_m\ne 0$, we have
\begin{equation} \label{E:eigen-cond}
 \text{\it $\eps_m \ne \eps_\beta$ for all $\beta\in\Comp(m)\setminus\{(m)\}.$}
 \end{equation}
Suppose also that $\rank(\Qmap_n )= \#\{\alpha\in\Comp(n)~|~\eps_\alpha\ne 0\}$.  Then $\Qmap_n$ is diagonalizable,  and by duality so is the operator $(X^\Qmap)_n$.
\end{theorem}

\begin{example}
For the $\Theta$-map we have $\eps_m = 0$ if $m$ is even and $\eps_\beta = 2^{\ell(\beta)} \ne 2 = \eps_m$ whenever $m$ is odd and $\beta\in\Comp(m)\setminus\{(m)\}$. The dimension of $\Theta(\QSym_n)$ equals the Fibonacci number $f_n$ (with $f_0 = f_1 = f_2 = 1$) \cite{Ste}, or the number of compositions of $n$ with only odd parts. Thus $\Theta_n$ is diagonalizable for all $n$.
\end{example}

\begin{remark}
The Hopf subalgebra $\Pi = \Theta(\QSym)$ of $\QSym$ is Stembridge's {\em peak algebra}, and the diagonalizability of the restriction of $\Theta$ to $\Pi_n$ was proved in \cite{BHW} by a direct argument. 

More generally, a technique for diagonalizing operators on $\Sol$ that correspond to various transformations of alphabets for noncommutative symmetric functions is described in \cite[\S3 \& Note~5.20]{NC2}. A prominent example is the $A \mapsto (1-q)A$ transform, which yields the operator $W \mapsto X^{\Theta} \cdot W$ when $q = -1$ (cf.\ \cite[Section~3]{BHT}). The eigenvectors are constructed from a unique family of Lie idempotents. Here we describe a way to recursively construct eigenvectors (for a broader family of operators) that turn out to be Lie quasi-idempotents and satisfy a uniqueness property. 
\end{remark}

Our construction of eigenvectors is as follows. For each $m$ such that $\eps_m\ne 0$ and \eqref{E:eigen-cond} holds, define $Z_m\in\Sol_m$ recursively by $\pair{Z_m}{M_m} = 1$ and 
\begin{equation}\label{E:eigen-recursion}
 \pair{Z_m}{M_\beta}  =  \frac{1}{\eps_m - \eps_\beta} \; \sum_{\alpha < \beta} 
\pair{Z_m}{M_\alpha} \pair{X^\Qmap \cdot X_\alpha}{M_\beta}  \text{ if $\beta\in\Comp(m)\setminus\{(m)\}$.}
\end{equation}
For $\alpha = a_1 \ldots a_k \in \Comp(n)$ such that $Z_{a_1}, \ldots, Z_{a_k}$ are all defined, define $Z_\alpha\in\Sol_n$ by
 \[Z_\alpha = Z_{a_1} \star \cdots \star Z_{a_k}.\]
Note that if $Z_\alpha$ is defined then $\eps_\alpha\ne 0$.
\begin{proposition}\label{P:eigenZ}
If $Z_\alpha$ is defined then
\begin{equation}\label{E:eigenZ}
X^\Qmap \cdot Z_\alpha = \eps_\alpha \, Z_\alpha.
\end{equation}
\end{proposition}
\begin{proof}
Since $X^\Qmap$ is dual to $\Qmap$, it is an algebra map with respect to $\star$ and so $X^\Qmap \cdot Z_{a_1 \ldots a_k} = (X^\Qmap \cdot Z_{a_1}) \star \cdots \star (X^\Qmap \cdot Z_{a_k})$. Therefore it suffices to prove that $X^\Qmap \cdot Z_n = \eps_n \, Z_n$, or equivalently that $\pair{X^\Qmap \cdot Z_n}{M_\beta} = \eps_n \pair{Z_n}{M_\beta}$ for all $\beta \in \Comp(n)$. This is done as follows:
\begin{eqnarray*}
 \pair{X^\Qmap \cdot Z_n}{M_\beta} 
 & = & \pair{X^\Qmap \cdot \sum_{\alpha \in \Comp(n)} \pair{Z_n}{M_\alpha} X_\alpha}{M_\beta} \\
 & = & \sum_{\alpha \le \beta} \pair{Z_n}{M_\alpha} \pair{X^\Qmap \cdot X_\alpha}{M_\beta} \\
 & = & \eps_\beta \pair{Z_n}{M_\beta} + \sum_{\alpha < \beta}
      \pair{Z_n}{M_\alpha} \pair{X^\Qmap \cdot X_\alpha}{M_\beta} \\
 & = & \eps_\beta \pair{Z_n}{M_\beta} + (\eps_n - \eps_\beta) \pair{Z_n}{M_\beta} \\
 & = & \eps_n \pair{Z_n}{M_\beta}.
\end{eqnarray*}
In the second equality, the sum is over $\alpha\le\beta$ because the operator $X^\Qmap$ is triangular in the $X$-basis (see Remark~\ref{R:triangular}).
\end{proof}

To finish the proof of Theorem~\ref{T:eigen}, note that because of our assumption about $\rank(\Qmap_n)$, it suffices to exhibit a set of linearly independent eigenvectors with cardinality equalling the number of nonzero eigenvalues counted with multiplicities. If \eqref{E:eigen-cond} holds for every $m \le n$ such that $\lambda_m\ne 0$,  then $Z_\alpha$ is defined whenever $\eps_\alpha\ne 0$. Moreover, the $Z_\alpha$'s are linearly independent because of their triangular relation to the $X$-basis. This completes the proof.

\begin{remark}
The recursion \eqref{E:eigen-recursion} was designed to make the calculation in the previous proof go through. However, it is easy to see that \eqref{E:eigen-recursion} is also a necessary condition for $Z_n$ to be an eigenvector with eigenvalue $\eps_n$.  The following makes this precise.
\end{remark}

\begin{proposition}
Let $\Qmap\in\End(\QSym)$ and $m\ge 0$. Suppose that $\eps_m\ne 0$ and
\eqref{E:eigen-cond} holds, and that $Z'_m$ is any element of $\Sol_m$ satisfying $X^\Qmap \cdot Z'_m = \eps_m\, Z'_m$ and $\pair{Z'_m}{M_m}\ne 0$. Then $Z'_m = \pair{Z'_m}{M_m} Z_m.$
\end{proposition}

\begin{example} Assuming $Z_1, Z_2,$ and $Z_3$ are defined, we have
\[ Z_1 = X_1, \quad \quad Z_2 = X_2 - \frac{1}{2} X_{11}, \]
\[ Z_3 = X_3 + \frac{u_{12}}{u_3-u_1u_2} \, X_{12} + \frac{u_1 u_2 - u_3 - u_{12}}{u_3 - u_1 u_2} \,X_{21} + \frac{1}{3} X_{111}.\]
\end{example}

\subsection{Primitive elements}
An element $W\in\Sol_n$ is called a {\it primitive element} if $\Delta_\Sol(W) = 1\otimes W + W\otimes 1$. If $W$ is primitive then it is a Lie quasi-idempotent in the sense of \cite[Theorem~3.1]{NC2} (cf.\ \cite[Theorems~3.1 and 3.2]{R}).

\begin{proposition}
Let $\Qmap\in\End(\QSym)$ and $n\ge 0$, and suppose that $Z_m$ is defined for all $m\le n$. Then $Z_n$ is a primitive element.
\end{proposition}
\begin{proof}
Let $(Z^*_\alpha)_{\alpha\in\Comp(n)} \subseteq \QSym$ denote the dual basis of $(Z_\alpha)_{\alpha\in\Comp(n)}$, so that $\pair{Z_\alpha}{Z^*_\beta} = \delta_{\alpha,\beta}$. Then $(Z^*_\alpha)$ form a basis of eigenvectors for $\Qmap_n$.  For any pair of compositions $\alpha, \beta$ such that $|\alpha|+|\beta| = n$, we have 
\[\pair{\Delta_\Sol(Z_n)}{Z^*_\alpha \otimes Z^*_\beta} = \pair{Z_n}{Z^*_\alpha Z^*_\beta}\]
by the Hopf algebraic duality between $\Sol$ and $\QSym$. Notice that $Z^*_\alpha Z^*_\beta$ is an eigenvector of $\Qmap$ with eigenvalue $\eps_{\alpha} \eps_{\beta} = \eps_{\alpha\cdot\beta}$. This eigenvalue equals $\eps_n$ if and only if $\alpha = (n)$ or $\beta = (n)$. Therefore the expansion of $Z^*_\alpha Z^*_\beta$ in the basis $(Z^*_\gamma)$ involves the element $Z^*_n$ if and only if $\alpha$ or $\beta$ is $(n)$. 
\end{proof}

In \cite[Theorem~3.16]{NC2} it is shown how to obtain a family of orthogonal idempotents of the descent algebra from a sequence of Lie quasi-idempotents. In particular, such a family can be obtained from our sequence $(Z_n)_{n\ge 0}$. See \cite[p.\ 919 Remark]{Brown} for a related result on obtaining orthogonal idempotents of the descent algebra.


\section{Describing endomorphisms by explicit matrices}\label{explicit-matrix-section}
We will describe how endomorphisms of $\QSym$ can be represented by triangular matrices with polynomial entries. An algorithm is suggested for computing these polynomials.

\subsection{Matrices representing endomorphisms}
A composition is called {\em Lyndon} if it is lexicographically smaller than all of its nontrivial cyclic rearrangements. Let $\mathcal{L}$ denote the set of Lyndon compositions. The set $\{ M_\alpha~|~\alpha \in \mathcal{L}\}$ freely generates $\QSym$ as an algebra, that is, $\QSym = \field[M_\alpha~|~\alpha\in\mathcal{L}]$
 \cite[\S 6, Example~1]{H}
(cf.\ \cite[Corollary~2.2]{MR}). For example,
the composition $21$ is not Lyndon, but we can express $M_{21}$ uniquely as a polynomial  in the monomial functions indexed by Lyndon compositions:
\[M_{21} = M_1 M_2 - M_{1 2} - M_3.\]
We call this the Lyndon expansion of $M_{12}$.

Recall from \eqref{E:univ} that the entries of the matrix for $\Qmap$ relative to the monomial basis are polynomials in the values $\pair{X^\Qmap}{M_\alpha}$. Since $\pair{X^\Qmap}{\cdot} : \QSym \to \field$ is an algebra map, $\pair{X^\Qmap}{M_\alpha}$ is a polynomial in the values $\pair{X^\Qmap}{M_\beta}$ for $\beta\in\mathcal{L}$. To describe these polynomials,
we introduce commutative variables $u_\alpha, \alpha\in\mathcal{L}$ and define $A_n(\beta,n)$, where $\beta \in \Comp(n)$, to be the polynomial obtained by substituting $u_\alpha$ for $M_\alpha$ in the Lyndon expansion of $M_\beta$.
For instance, $A(21,3) = u_1 u_2 - u_{12} - u_3$.

For  $n\ge 1$ and $\alpha,\beta\in\Comp(n)$, let  $A_n(\beta, \alpha) = 0$ if $\alpha\not\le\beta$ and
\[ A_n(\beta, \alpha) = A_{a_1}(\beta_1,a_1) A_{a_2}(\beta_2, a_2) \cdots A_{a_k}(\beta_k, a_k) \]
if $\alpha\le\beta$, where $\beta_1, \ldots, \beta_k$ are the compositions such that $\beta = \beta_1 \cdots \beta_k$ and $\alpha = (|\beta_1|,|\beta_2|,\ldots, |\beta_k|)$.  In addition, we set $A_0 = (1)$. The matrices $A_1, A_2$, and $A_3$ are
\[ A_1 = \begin{tabular}{c|c}
& $1$ \\
\hline
$1$ & $u_1$ 
\end{tabular}
\;\;\;\;\;\;\;\;\;\;
A_2 = 
\begin{tabular}{c|cc}
   & $2$ & $11$  \\
   \hline 
     $2$ &  $u_2$ & $0$  \\
    $11$ & $\frac{1}{2} (u_1^2 - u_2)$ & $u_1^2$
\end{tabular} 
\]
\begin{equation} \label{E:d4matrix}
A_3 = \begin{tabular}{c|cccc}
    & $3$ & $12$ & $21$ & $111$  \\
    \hline 
      $3$ &  $u_3$ & $0$ & $0$ & $0$  \\
      $12$ & $u_{12}$ & $u_1 u_2$ & $0$ & $0$ \\
      $21$ & $u_1 u_2 - u_{12} - u_3$ & $0$ &  $u_1 u_2$ &  $0$ \\
      $111$ & $\frac{1}{6} u_1^3 - \frac{1}{2} u_1 u_2 + \frac{1}{3} u_3$ & 
          $\frac{1}{2} (u_1^2 - u_2) u_1$ & $\frac{1}{2}(u_1^2 - u_2) u_1$ & 
          $u_1^3$ 
\end{tabular}
\end{equation}

The preceding discussion leads to this refinement of Proposition~\ref{P:universal}:
\begin{theorem}\label{T:matrix}
For any assignment of real numbers to the variables $u_\alpha$, the linear operator $\Qmap$ on $\QSym$ defined by
\begin{equation}\label{E:matrix}
\Qmap(M_\beta) = \sum_{\alpha \le \beta} A_n(\beta,\alpha) M_\alpha \;\;\;\text{for all $n\ge 0$ and $\beta\in\Comp(n)$}
\end{equation}
is a endomorphism of $\QSym$. Moreover, every endomorphism of $\QSym$ has the form \eqref{E:matrix}.
\end{theorem}

\begin{example}\label{E:construct-endo}
We apply Theorem~\ref{T:matrix} to construct an endomorphism that satisfies the nonnegativity hypotheses of Theorem~\ref{T:end-markov} and hence gives rise to a $QS^*$-distribution. 

Let $c_{\alpha,\beta}$ be as in \eqref{E:cdef}. By a change of basis,
\[ c_{\alpha,n} = \sum_{\beta\ge\alpha} A_n(\beta,n). \]
For $n=3$, this gives
\[ \begin{array}{ll}
c_{3,3} = \frac{1}{6} u_1^3 + \frac{1}{2} u_1 u_2 + \frac{1}{3} u_3 \quad & 
 c_{12,3} = \frac{1}{6} u_1^3 - \frac{1}{2} u_1 u_2 + u_{12} + \frac{1}{3} u_3 \\
c_{21,3} = \frac{1}{6} u_1^3 + \frac{1}{2} u_1 u_2 - u_{12} - \frac{2}{3} u_3 \quad &
 c_{111,3} = \frac{1}{6} u_1^3 - \frac{1}{2} u_1 u_2 + \frac{1}{3} u_3 
 \end{array} \]
Let us choose a specialization of the variables so that the $c_{\alpha,3}$ are nonnegative and not identically zero. For instance, set $u_1 = 2, u_2 = 1/2, u_3 = 2, u_{12} = -1$, and assign arbitrary values to all other $u_\alpha$. Then $c_{3,3} = 5/2$, $c_{12,3} = 1/2$, $c_{21,3} = 3/2$, $c_{111,3} = 3/2$, and the corresponding $QS^*$-distribution is (using \eqref{E:QSf-simplified})
\[ \sum_{\sigma\in\perm_3} \QSf_\Qmap(\sigma) \sigma = \frac{1}{2^3} \cdot \sum_{\alpha\in\Comp(3)} c_{\alpha,3} \; Y_\alpha = \frac{5}{16} \cdot {\bf 123} + \frac{1}{16} \cdot ({\bf 213}+{\bf 312}) + \frac{3}{16} \cdot ({\bf 132} +  {\bf 231} + {\bf 321}).\]
We can quickly find the eigenvalues for the transition probability matrix $\Kbar$ of the lumped random walk by reading off the diagonal entries of $A_n$ and multiplying by $1/2^3$; they are $1, 1/4, 1/8, 1/8$.
\end{example}

\subsection{Algorithm for computing $A_n$}
One might hope for a reasonable algorithm for computing the polynomials appearing in the matrix $A_n$, which amounts to computing the polynomial expansion of an arbitrary $M_\alpha$ in terms of  those $M_\beta$ with $\beta\in\mathcal{L}$.   However, the Lyndon expansion for $M_\alpha $ is often exponential in length, precluding the existence of an efficient algorithm.

An inefficient recursive algorithm is as follows: for $\alpha = (a_1,\dots , a_l)$ choose the earliest $k$ such that $(a_k, a_{k+1},\dots , a_{l-1}, a_l)$ is Lyndon.  Then  $M_{(a_1,\dots ,a_{k-1})}M_{(a_k,\dots , a_l )}$ is a sum over all ways of quasi-shuffling 
$(a_1,\dots , a_{k-1})$ with $(a_k,\dots , a_l)$, in the sense of \cite{H}. This gives a straightening law on monomial quasisymmetric functions which may be used to reduce any monomial quasisymmetric function into a sum of products of  Lyndon ones in a finite number of steps;  one summand in $M_{(a_1,\dots , a_{k-1})}M_{( a_k,\dots , a_l )}$
is $M_{\alpha }$ while all other summands either have strictly fewer parts than 
$\alpha $ or else 
are lexicographically smaller of the same length as $\alpha $; in either case,
they are closer to Lyndon. 

%
%

\section{Connections to the QS-distribution and random walks of Bidigare, Hanlon, and Rockmore}\label{QS-section}

Throughout this section, let $(r_1, r_2, \ldots)$ be an infinite sequence of nonnegative real numbers summing to $1$. The {\it $QS$-distribution} on $\perm_n$ may be defined \cite[Theorem~2.1]{QS} as the probability distribution on $\perm_n$ in which a permutation $\pi$ is selected with probability 
\[\prob_{QS}(\pi) = F_{\Des(\pi^{-1})}(r_1, r_2, \ldots).\]

The link between the $QS$-distribution and the $QS^*$-distribution comes from the observation that the evaluation map on $\QSym$ sending $x_i$ to $r_i$ for all $i$ is a character, and so by Proposition~\ref{P:transpose} there is a unique $\Qmap\in\End(\QSym)$ such that 
\begin{equation}\label{E:evaluation}
\pair{X^\Qmap}{G(x_1, x_2, \ldots)}= G(r_1, r_2, \ldots) \quad \text{for all $G\in\QSym$.}
\end{equation}
This leads to the following result:

\begin{theorem}\label{QS-theorem} 
There exists  $\Qmap\in\End(\QSym)$ such that for every permutation $\pi$,
 \begin{equation} \label{E:QS}
  \QSf_\Qmap(\pi) = \prob_{QS}(\pi^{-1})
 \end{equation}
In other words, the transition probability matrix of the left random walk on $\perm_n$ driven by the $QS$-distribution is $K^\Qmap_n$, with $K^\Qmap_n(\pi,\sigma\pi) = \prob_{QS}(\sigma)$ for all $\sigma,\pi\in\perm_n$.
\end{theorem}
\begin{proof} 
Choose $\Qmap$ so that \eqref{E:evaluation} holds. By \eqref{E:Xphi-c} and \eqref{E:QSf-simplified}, 
\[
F_{\Des(\pi)}(r_1, r_2, \ldots) = \pair{X^\Qmap}{F_{\Des(\pi)}}  = \eps^n \; \QSf_\Qmap(\pi).
\]
We have $\eps = \pair{X^\Qmap}{M_1} = M_1(r_1, r_2, \ldots) = \sum r_i = 1$, completing the proof.
\end{proof}

\begin{remark} The $QS$-distribution is an instance of a transformation of alphabets in the context of noncommutative symmetric functions \cite{NC2}. This viewpoint is developed in further in \cite[Section~3.6]{NC6}, where the main results of \cite{QS} are deduced from the theory of free quasisymmetric functions.
\end{remark}

Let us  describe more explicitly the entries of $K$ and $\Kbar$. Let $\delta_1, \delta_2, \ldots$ be i.i.d.\ random variables such that 
\[ \prob(\delta_i = j-1) = r_j \;\;\;\;\text{for $j$ a positive integer.}\]
Given a sequence of distinct numbers $\pi =(\pi_1, \ldots, \pi_n)$, let $\st(\pi)$ denote the standardization of $\pi$, that is, the unique permutation $(\sigma_1, \ldots, \sigma_n) \in \perm_n$ such that for all $i,j \le n$, $\pi_i < \pi_j$ if and only if $\sigma_i < \sigma_j$. For instance, $\st((-2,1,3,-4)) = (2,3,4,1).$ 

\begin{theorem} \label{T:Kmatrix-ashuffle}
Let $\Qmap$ be as in Theorem~\ref{QS-theorem}. For all $\sigma, \tau\in\perm_n$ and $\beta\in\Comp(n)$, we have
\[ K^\Qmap_n(\sigma,\tau) = \prob( \st((\sigma_1 + \delta_1 n, \sigma_2 + \delta_2 n, \ldots, \sigma_n + \delta_n n)) = \tau)\]
and
\[ \Kbar^\Qmap_n(\Des(\sigma), \beta) = \prob( \Des((\sigma_1 + \delta_1 n, \sigma_2 + \delta_2 n, \ldots, \sigma_n + \delta_n n)) = \beta).\]
\end{theorem}
\begin{proof}
Let $\pi = \sigma \tau^{-1}.$ By Theorem~\ref{QS-theorem},
\[ K^\Qmap_n(\sigma,\tau) =  F_{\Des(\pi)}(r_1, r_2, \ldots) = \sum_{i_1 \le i_2 \le \cdots \le i_n \atop i_k \in S_{\Des(\pi)} \implies i_k < i_{k+1}} r_{i_1} \cdots r_{i_n}.\]
In other words, $K(\sigma,\tau)$ is the probability that
\begin{equation}\label{E:delta1}
 \delta_1 \le \delta_2 \le \cdots \le \delta_n \;\;\;\text{ and } \;\;\; \delta_i < \delta_{i+1} \text{ whenever $i \in S_{\Des(\pi)}$.}
 \end{equation}
Now consider instead the probability that
$\st((\sigma_1 + \delta_1 n, \sigma_2 + \delta_2 n, \ldots, \sigma_n + \delta_n n))=\tau=\pi^{-1}\sigma.$ Let $i_1, \ldots, i_n$ be indices such that $\sigma_{i_j} = \pi_j$ for all $j$. Multiplying $\sigma$ on the left by $\pi^{-1}$ has the effect of replacing the occurrence of $\sigma_{i_j}$ in $(\sigma_1, \ldots, \sigma_n)$ by the value $j$. The resulting permutation equals $\st((\sigma_1 + \delta_1 n, \sigma_2 + \delta_2 n, \ldots, \sigma_n + \delta_n n))$ if and only if $\pi_1 + \delta_{i_1} n < \pi_2 + \delta_{i_2} n < \cdots < \pi_n + \delta_{i_n} n$, and this occurs if and only if
\begin{equation}\label{E:delta2}
\delta_{i_1} \le \delta_{i_2} \le \cdots \le \delta_{i_n} \;\;\;\text{and}\;\;\; \delta_{i_j} < \delta_{i_{j +1}} \text{ whenever $j \in S_{\Des(\pi)}$.}
 \end{equation}
Since the $\delta_i$'s are i.i.d.,  events \eqref{E:delta1} and \eqref{E:delta2} occur with equal probability. The formula for $\Kbar$ follows via lumping.
\end{proof}

The eigenvalues of $\Kbar$ can be determined immediately from Proposition~\ref{P:eigenvalues} (cf.\ \cite[Theorem~2.2]{QS} and references thereafter).

\begin{proposition}
Let  $\Qmap$ be as in Theorem~\ref{QS-theorem}. For any $n\ge 0$, the eigenvalues of $\Kbar$ are the power sum symmetric functions $p_\mu(r_1, r_2, \ldots)$ where $\mu$ ranges over all partitions of $n$. The multiplicity of $p_\mu(r_1, r_2, \ldots)$ is the number of different compositions obtainable by rearranging the parts of the partition $\mu$.
\end{proposition}

Bidigare, Hanlon, and Rockmore (BHR) introduced a class of random walks on chambers of central hyperplane arrangements \cite{BHR}, generalizing many shuffling and sorting schemes like the Tsetlin Library. Their work was developed further by Brown and Diaconis \cite{BD} and Brown \cite{Brown, Brown2}.  In some cases a BHR random walk is isomorphic to a right random walk on a finite reflection group. For instance Stanley showed that the transpose of the right random walk on $\perm_n$ driven by the $QS$-distribution is an instance of a BHR random walk on the chambers of the braid arrangement \cite{QS}. More generally, Theorem~8 of \cite{Brown} implies that if $W = \sum_{\sigma\in\perm_n} W(\sigma) \sigma$ is a probability distribution on $\perm_n$ such that $W\in \Sol_n$ and the expansion of $W$ in the basis $\{X_\alpha\}$ has nonnegative coefficients, then right multiplication by $W$ is isomorphic to a BHR random walk. Such a distribution $W$ will be called a {\it BHR-distribution.}

Not every $QS^*$-distribution is a BHR-distribution. The $\Theta$-map provides one family examples: for instance when $n=3$,
\[ (X^\Theta)_3 = 2^3 \cdot \sum_{\sigma \in \perm_3} \QSf_\Theta(\sigma)\sigma = 2 ( Y_3 + Y_{12} + Y_{111}) = 2 (X_3 - X_{21} + X_{111}).\]
The endomorphism $\Qmap$ considered in Example~\ref{E:construct-endo} provides another example:
\[ (X^\Qmap)_3 = 2^3 \cdot \left[\frac{5}{16} Y_3 + \frac{1}{16} Y_{12} + \frac{3}{16} (Y_{21} + Y_{111})\right] = 2 X_3 - X_{12} + \frac{3}{2} X_{111}. \]

Note that BHR random walks on $\perm_n$ are right random walks, whereas in this paper we are dealing with left random walks (encoded by $K$) and their lumped versions (encoded by $\Kbar$).  
This distinction might not matter much in practice,  since there is no difference between left and right random walks if we start at the identity permutation. Also, all of the associated matrices--left or right, lumped or not--that arise from a BHR-distribution on $\perm_n$ have the same eigenvalues, thanks to \cite[Theorem~8]{Brown} (cf.\ \cite[Theorem~3.12]{NC2}).

The references \cite{BHR, BrD, Brown} all give bounds on rates of convergence to the stationary distribution. Here we will not pursue the problem of estimating the convergence rates of $K$ or $\Kbar$,  although it would be interesting to see whether any techniques from those papers could be adapted to our setting.

\section{Connections to $a$-shuffles and the Tchebyshev operator}\label{a-section}

\subsection{a-shuffles}\label{ashuffle-section} Let $a$ be a positive integer. An {\it $a$-shuffle} of a deck of cards involves cutting the deck into $a$ (possibly empty)
packets according to the multinomial distribution and then letting cards fall one at a time from the bottoms of the packets into new pile, where the probability that the bottom card from a particular packet falls is proportional to the current size of the packet. When $a = 2$ this is just the standard Gilbert-Shannon-Reeds (GSR) model of riffle shuffling. See \cite{BD} details. 

A well-known formula due to Bayer and Diaconis \cite{BD} for the probability that the deck is in arrangement $\pi$ after an $a$-shuffle is
\begin{equation}\label{E:BD}
  \prob_a(\pi) = \binom{n+a-d(\pi^{-1}) - 1}{n}/a^n,
 \end{equation}
where $d(\pi)$ is the number of descents in $\pi$. 

 Let $\Psi_a$ denote the unique endomorphism of $\QSym$ such that $X^{\Psi_a} =  X \star \cdots \star X$ ($a$ terms), where $X$ is the universal character introduced in Section
 ~\ref{S:random-walks}. 
 This convolution character is discussed in \cite[Example~4.7]{ABS}, where an explicit formula for $\Psi_a$ is given.

\begin{proposition}\label{P:ashuffle}
  For every permutations $\pi$, 
\begin{equation}
\prob_{\Psi_a}(\pi) = \prob_{a}(\pi^{-1}) = \prob_{QS}(\pi^{-1}),
\end{equation} 
where the $QS$-distribution has $r_1 = r_2 = \cdots = r_a = \frac{1}{a}$ and $r_i = 0$ for $i>a$.
\end{proposition}
\begin{proof}
The second equality was shown by Stanley \cite{QS}. Using \eqref{E:univ} it is easy to obtain a formula for $\Psi_a$ in the monomial basis (or see \cite[(4.5)]{ABS}), from which it follows that 
\[\Kbar^{\Psi_a}_n(\beta,n) = \binom{n + a - \ell(\beta)}{n}/ a^n.\]
This is equivalent to the formula of Bayer and Diaconis, as $\ell(\Des(\pi)) = d(\pi) + 1.$
\end{proof}

Proposition~\ref{P:ashuffle} can also be inferred from \cite[\S3.6]{NC6}.

The eigenvalues and diagonalizability of operators associated with $a$-shuffles are well known; see, for instance, \cite{BD, BHR, Brown}. Noting that  $\pair{X^{\Psi_a}}{M_n} = a$, the results of Section~\ref{recursive-section} imply the following:

\begin{proposition}\label{P:Psi-eigenval}
For all $n\ge 0$, $(\Psi_a)_n$ is diagonalizable. Its eigenvalues, with multiplicities, are $(a^{\ell(\alpha)})_{\alpha\in\Comp(n)}$.
\end{proposition}

It is also possible to determine an explicit formula for the eigenvectors:

\begin{proposition}\label{P:Psi-eigenvec}
For every $n\ge 0$, the unique eigenvector $Z_n$ such that $X^{\Psi_a} \cdot Z_n = a Z_n$ and $\pair{Z_n}{M_n} = 1$ is given by 
\begin{equation} \label{E:shuffle-Z}
Z_{n} = \sum_{\beta\in\Comp(n)} \frac{(-1)^{\ell(\beta) -1}}{\ell(\beta)} X_{\beta}.
\end{equation}
\end{proposition}
\begin{proof}
For compositions $\alpha\le \beta$ as above, let
\[g(\beta,\alpha) = \binom{a}{\ell(\beta_1)} \cdots \binom{a}{\ell(\beta_k)}\]
where $\beta = \beta_1 \cdots \beta_k$ and $\alpha = (|\beta_1|, \ldots, |\beta_k|)$. It follows from \eqref{E:univ} that $\pair{X^{\Psi_a} \cdot X_\beta}{M_\alpha} = g(\beta,\alpha)$. If \eqref{E:shuffle-Z} holds then
\[
X^{\Psi_a}\cdot Z_n = \sum_{\beta \in\Comp(n)} \frac{(-1)^{\ell(\beta)-1}}{\ell(\beta)} 
\sum_{\gamma\ge\beta} g(\gamma,\beta) X_\gamma 
\nonumber
= \sum_{\gamma \in\Comp(n)} X_\gamma \left( \sum_{\beta\le\gamma} 
 \frac{(-1)^{\ell(\gamma)-1}}{\ell(\gamma)} g(\gamma,\beta)\right).
\]
So, to conclude that $X^{\Psi_a}\cdot Z_n = a Z_n$ it suffices to prove
\begin{equation}\label{E:Psidual3}
 \sum_{\beta\le\gamma} 
 \frac{(-1)^{\ell(\gamma)-1}}{\ell(\gamma)} g(\gamma,\beta) =  a \, \frac{(-1)^{\ell(\gamma)-1}}{\ell(\gamma)}.
\end{equation}
We may assume without loss of generality that  $\gamma = 1^n.$ In this case \eqref{E:Psidual3} becomes
\[\sum_{(c_1,\ldots, c_h) \comp n} \frac{(-1)^{h-1}}{h} \binom{a}{c_1} \cdots \binom{a}{c_h} =  a \, \frac{(-1)^{n-1}}{n},\]
which is verified by equating the coefficients of $x^n$ in
\[\ln(1+((1+x)^a -1)) = a \ln(1+x).\]
\end{proof}

\begin{remark} Under the isomorphism between $(\Sol, \star, \Delta_\Sol)$ and the Hopf algebra of noncommutative symmetric functions, $n\cdot Z_n$ goes to the noncommutative power sum symmetric function of the second kind indexed by $n$ (see \cite{NC1} for definitions).
\end{remark}

Let $\{P_\alpha\}\subseteq \QSym_n$ be the dual basis of $\{Z_\alpha\}$. Thus, $\pair{Z_\alpha}{P_\beta} = \delta_{\alpha,\beta}$ and
\[\Psi_a(P_\beta) = a^{\ell(\beta)} P_\beta.\] This basis was introduced in \cite{MR}. An explicit formula for $P_\alpha$ is obtained as follows.
 For compositions $\alpha = (a_1, \ldots, a_k)$ and $\beta = (b_1,\ldots, b_\ell)$ such that $\alpha\le\beta$, let $\beta_1, \beta_2, \ldots, \beta_k$ be the sequence of compositions such that 
$\beta = \beta_1\cdots \beta_k$ and $|\beta_i| = a_i$ for all $i$. Let $f(\beta,\alpha) =  \ell(\beta_1)! \cdots \ell(\beta_k)!.$
Then we have
\begin{equation}\label{E:PM}
P_\beta = \sum_{\alpha\le \beta} \frac{1}{f(\beta,\alpha)} M_\alpha.
\end{equation}
This is \cite[Formula~(2.12)]{MR}. 

\subsection{Tchebyshev operator of the second kind}

The Tchebyshev operator (of the second kind) was introduced by Hetyei \cite{Het} as an operator on $ab$-words that encodes how the flag $f$-vector of a  graded poset changes when the poset undergoes a certain combinatorial transformation. The Tchebyshev polynomials of the second kind can be obtained by a suitable specialization of the operator. Ehrenborg and Readdy \cite{ER} introduced an equivalent operator $U$ on $\QSym$ and showed that it is a Hopf endomorphism. They showed that $U$ is diagonalizable and determined the eigenvalues and eigenvectors. Here we explain how the spectra can be deduced from our results and how $U$ encodes riffle shuffles.

Given a graded poset $P$ of rank $n\ge 0$ with unique minimal and maximal elements $\hat{0}$ and $\hat{1}$, respectively, define $F(P) \in \QSym_n$ by $F(P) = 1$ if $n=0$ and
\[ F(P) = \sum_{\hat{0} = t_0 < t_1 < \cdots < t_k = \hat{1}} M_{(\rho(t_0,t_1), \rho(t_1,t_2),\ldots,\rho(t_{k-1},t_k))}\]
if $n\ge 1$, where the sum is over all chains in $P$ from $\hat{0}$ to $\hat{1}$. The definition of $F(P)$ is due to Ehrenborg \cite{E96}. See \cite{EC1} for background on posets. The {\it Tchebyshev operator of the second kind} \cite{ER} is the linear operator $U$ on $\QSym$ satisfying
\[
 U(F(P)) = \sum_{\hat{0} = t_0 < t_1 < \cdots < t_k = \hat{1}} \pair{G}{F({[t_0,t_1]})} \cdots \pair{G}{F({[t_{k-1},t_k]})} \cdot M_{(\rho(t_0, t_1), \ldots, \rho(t_{k-1}, t_k))}
\]
for every graded poset $P$, where $G\in\dhat$ is the character given by $\pair{G}{F(P)} = \# P$. 
\begin{proposition}
We have
\begin{equation}
 U = \Psi_2.
 \end{equation}
\end{proposition}
\begin{proof}
It is straightforward to check that  $G = X \star X$, using the fact that $P\mapsto F(P)$ defines a Hopf algebra homomorphism from the Hopf algebra of graded posets to $\QSym$ \cite{E96}  and that $\QSym$ is spanned by the $F(P)$ as $P$ ranges over all graded posets \cite{BL}. Comparing \eqref{E:univ} with the definition of $U$, we get $U\circ F = {\Psi_2} \circ F$, which implies $U = \Psi_2$ since the $F(P)$ span $\QSym$.
\end{proof}

Formulas for the eigenvalues and eigenvectors are given in Proposition~\ref{P:Psi-eigenval} and \eqref{E:PM} (cf.\ \cite[Theorem~10.10]{ER}). 
Also, Proposition~\ref{P:ashuffle} yields the following probabilistic interpretation:
\begin{corollary}
The coefficient of $F_\alpha$ in the fundamental-basis expansion of $\frac{1}{2^{nk}} U^k(F_n)$ is the probability of ending up with some permutation with descent composition $\alpha\in\Comp(n)$ after performing $k$ riffle shuffles.
\end{corollary}

A natural question is whether there is a transformation of graded posets for which the operator $\Psi_m$, $m> 2$, plays the role analogous to the Tchebyshev operator $U$.

\section{A random walk on $ab$-words}\label{S:ab}

The $cd$-index is a remarkably convenient encoding for the flag $f$-vector of any Eulerian poset (e.g. Bruhat order), namely the vector counting chains through various rank sets (e.g. flags of faces of  specified dimensions in a regular cell decomposition of a sphere).   The $ab$-index is a vector which makes  sense for all graded posets, and has as its basis all words in the two noncommuting variables $a$ and $b$.  The $cd$-index is derived from the $ab$-index by the substitutions $c=a+b$ and $d=ab+ba$.  There is an extensive literature using $cd$-index to study which  vectors may arise as flag $f$-vectors.  See e.g.\ \cite{BER} for further background.

In \cite{E04}, Ehrenborg introduced a family of linear operators $\omega_r$ which act on $ab$-words as follows.  Given an $ab$-word, replace each appearance
of $ab$ by $r(ab - (r-1)ba)$.  Now send each remaining $a$ from the
original word to $a + (r-1)b$ and send each remaining $b$ from
the original word to $b + (r-1)a$.  This operator, which generalizes the $\omega$ map of \cite{BER}, was designed to compute flag vectors of the $r$-Birkhoff transform, perhaps most notably including a family of complex regular  polytopes known as Shephard's generalized orthotopes. Ehrenborg showed that the operator $r \omega_r$ is isomorphic to a Hopf algebra endomorphism $\vartheta_r \in \End(\QSym)$ and he made some conjectures about the spectrum. This section proves Ehrenborg's conjectures. We will then show that $\vartheta_r$ is dual to the $q$-bracketing operators studied in \cite{NC2}. Finally, we give a probabilistic interpretation of $\vartheta_r$ that generalizes a theorem of Stembridge \cite{Ste}.

First let us define $\vartheta_r$. Given an $ab$-word $u=u_1 u_2 \cdots u_n$, where each $u_i$ is $a$ or $b$, we define $S_u \subseteq [n]$ by putting $i\in S_u$ if and only if $u_i=b$. For instance $S_{abaab} = \{2,5\}$. The correspondence $u\leftrightarrow S_u$ is a bijection between $ab$-words of length $n$ and subsets of $[n]$. This bijection induces an isomorphism, denoted $\gamma$, from the vector space of $ab$-words onto $\bigoplus_{i=1}^\infty \QSym_i$, given by $\gamma(u)= F_{\co(S_u)}$. Define  $\vartheta_r$ to be the linear operator on $\QSym$ given by $\vartheta_r(F) = \gamma ( r \omega_r ( \gamma^{-1} (F) ))$ if $F \in \bigoplus_{i=1}^\infty \QSym_i$ and $\vartheta_r(1) = 1$. Ehrenborg \cite{E04} proved that $\vartheta_r$ is a Hopf algebra endomorphism of $\QSym$ and that $\vartheta_2$ is  Stembridge's map $\Theta$, and he made the follow conjecture on the spectrum: 

\begin{proposition} \label{Ehrenborg-conjectures}
For all $n\ge 0$ and $r\in\R$, $(\vartheta_r)_n$ is diagonalizable and its eigenvalues, counting multiplicities, are
\[ \eps_\alpha = (1-(1-r)^{a_1})(1-(1-r)^{a_2})\cdots (1-(1-r)^{a_k})\]
for $\alpha = (a_1, a_2, \ldots, a_k)\in\Comp(n)$.
\end{proposition}
\begin{proof}
For any $ab$-word $u$, $r^n \Kbar(\co(S_u),n)$ is the coefficient of $a^{n-1}$ in $r \cdot \omega_r(u)$ (here $\Kbar = \Kbar^{\vartheta_r}_n$). This coefficient is straightforward to compute using the definition of $\omega_r$: it  is $r (r-1)^k$ if $u$ has the form $b^k a^{n-1-k}$ and $0$ otherwise. It follows that
\begin{equation}\label{E:Kthetar}
 \Kbar(\alpha,n) = \begin{cases} 
\frac{\cdot (r-1)^k}{r^{n-1}} & \text{if $\alpha = (1^k, n-k)$ for $0\le k \le n-1$,} \\
0 & \text{otherwise.}
\end{cases}
\end{equation}
Note that
\[
\vartheta_r(M_n) = \vartheta_r\left(\sum_{\alpha\in\Comp(n)} (-1)^{\ell(\alpha)-1} F_\alpha\right) 
= \sum_{\gamma \in\Comp(n)} M_\gamma \sum_{\alpha\in\Comp(n) \atop \beta \le \gamma} (-1)^{\ell(\alpha)-1} r^n \Kbar(\alpha,\beta).
\]
Since $\eps_n$ is the coefficient of $M_n$ in the 
monomial expansion of $\vartheta_r(M_n)$, we have
\[ \eps_n = \sum_{\alpha\in\Comp(n)} (-1)^{\ell(\alpha) -1} r^n \Kbar(\alpha,n) = \sum_{k = 0}^{n-1} (-1)^k r (r-1)^k = 1-(1-r)^n.\]
The formula for the eigenvalues now follows from Proposition~\ref{P:eigenvalues} and the diagonalizability follows from Theorem~\ref{T:eigen}. 
\end{proof}

\begin{remark}
If $r\ge 1$ and $\pi\in\perm_n$ then  $\QSf_{\vartheta_r}(\pi)$ is the probability that a deck of $n$ cards is in arrangement $\pi$ after performing an inverse $r$-weighted face-up face-down shuffle; that is, remove  a subset of cards from the deck to form a new packet, where a card is selected independently for removal with probability $1-\frac{1}{r}$ and the cards are kept in the same relative order, and then place the new packet face down on top of the pile of remaining cards.
\end{remark}

In \cite{NC2} it was show that the $A\mapsto (1-q)A$ transform on noncommutative symmetric functions is equivalent to the left action on the descent algebra $\Sol_n$ by 
\[ \eta_q = (1-q) \cdot \sum_{k=0}^{n-1} (-q)^k Y_{(1^k,n-k)}. \]
This action was shown to be diagonalizable and formulas for the eigenvalues and eigenvectors were given. The eigenvalues turn out to be the same as the eigenvalues for $\vartheta_{1-q}.$ This is explained by the following result.

\begin{proposition}
The restriction of $\vartheta_r$ to $\QSym_n$ is dual to the left action of $\eta_{1-r}$ on $\Sol_n$; that is, 
\[\pair{\eta_{1-r} \cdot W}{G} = \pair{W}{\vartheta_r(G)}\] for all $W\in \Sol_n$ and $G \in \QSym_n$. In other words, $(X^{\vartheta_r})_n = \eta_{1-r}.$
\end{proposition}
\begin{proof}
Using the formula for $\Kbar(\alpha,n)$ derived in the proof of Proposition~\ref{Ehrenborg-conjectures}, we have \[\eta_{1-r} = r \sum_{k=0}^{n-1} (r-1)^k Y_{(1^k, n-k)} = \sum_{\alpha\in\Comp(n)}\Kbar (\alpha,n) Y_\alpha =( X^{\vartheta_r})_n.\] Applying Proposition~\ref{P:transpose} completes the proof.  

This result may also be proven by constructing an explicit isomorphism from $\vartheta_r$ to the dual of $\eta_{1-r}$. Consider an $ab$-word upon which Ehrenborg's operator acts. Send each ``$a$'' to a descent and send each ``$b$'' to an ascent.  
Send $r$ to $1-q$.  Each ``$ab$" now corresponds to a valley, and for each set $S$, the allowable sets $T$ in the image are as follows.  Each valley must be
sent either to a valley or to a peak for a set to be allowable, but  this is the only requirement.  Now for each $T$, the coefficient $r^m(r-1)^n$ is obtained as follows.  The exponent $m$ counts valleys in $S$, just as the exponent for $1-q$ in $\eta_q$ does. The exponent $n$ counts  discrepancies between $S$ and $T$, with each valley/peak combination counted as a single discrepancy.  Thus, we obtain the following formula, which is the dual version of the formula for the operator $\eta_q$ proven in \cite[Proposition~5.41]{NC2}:
\[ 
\vartheta_{1-q} (F_\beta) = (1-q)^{hl(\beta)} \sum_{\Lambda (\beta) \subseteq S_\alpha \Delta (S_\alpha+1)}
(-q)^{b(S_\alpha,S_\beta)} F_{\alpha},\]
where $\beta$  can be written as a concatenation $\beta = \beta_1\cdots \beta_{hl(\beta )}$ of hook compositions $\beta_i = (1^k,l)$ and $b(S,T) = |(1 + (S\setminus T)) \cup (T \setminus S)|$.
\end{proof}

Given $r>1$, let $\delta_{r,1}, \delta_{r,2},\ldots$ be i.i.d.\  random variables taking on values in $\{1,-1\}$ such that
\[\prob(\delta_{r,i} = 1) = \frac{1}{r}, \;\;\;\text{ and } \;\;\;
 \prob(\delta_{r,i} = -1) = 1 - \frac{1}{r}.\]

\begin{theorem} \label{T:prob-q}
For all $\sigma,\tau \in \perm_n$ and $\beta\in\Comp(n)$, we have
\[
K^{\vartheta_r}_n(\sigma,\tau) = \prob(\st((\delta_{r,1} \sigma_1,\ldots, \delta_{r,n} \sigma_n)) = \tau)
\]
and
\[ \Kbar^{\vartheta_r}_n(\Des(\sigma),\beta) = \prob(\Des((\delta_{r,1} \sigma_1,\ldots, \delta_{r,n} \sigma_n)) = \beta). \]
\end{theorem}
\begin{proof}
Left multiplying $\sigma $ by $Y_{(1^k, n-k)}$ has the effect of negating all possible choices of $k-1$ values chosen from $[2,n]$ within the permutation $\sigma $, then standardizing values to obtain a permutation, since this negation reverses  relative order of the values being negated and makes them smaller than
all other values. Notice also that negating the value $1$ has no impact on the relative order of values. The sum 
\[\sum_{k=0}^{n-1} (r-1)^k Y_{(1^k, n-k)}\]
may therefore be viewed as a sum over all choices of which values in $[2,n]$ to negate, with $k$ recording the number of values other than $1$ being negated.  If each value in $[n]$ is independently negated with  probability $1- \frac{1}{r}$, then
\[\sum_{k=0}^{n-1} \frac{(r-1)^k(1)^{n-1-k}}{r^{n-1}} Y_{(1^k, n-k)}\]
is a sum over all possible subsets of $[n]$ to be negated, with each possibility  multiplied by its probability of being chosen.
\end{proof}

Setting $r=2$ in Theorem~\ref{T:prob-q} yields Theorem~3.6 of Stembridge \cite{Ste}.


\section{acknowledgments}
The authors thank Persi Diaconis and Phil Hanlon for helpful comments and suggestions.

 \end{document}